\documentclass{amsart}
\usepackage{mathrsfs}
\usepackage{mathrsfs}
\usepackage{amsfonts}
\usepackage{cases}
\usepackage{latexsym}
\usepackage{amsmath}
\usepackage[all]{xy}
\usepackage{stmaryrd}
\usepackage{amsfonts}
\usepackage{amsmath,amssymb,amscd,bbm,amsthm,mathrsfs,dsfont}

\usepackage{fancyhdr}
\usepackage{amsxtra,ifthen}
\usepackage{verbatim}

\numberwithin{equation}{section}

\theoremstyle{plain}
\newtheorem{theorem}{Theorem}[section]
\newtheorem{lemma}[theorem]{Lemma}
\newtheorem{proposition}[theorem]{Proposition}
\newtheorem{corollary}[theorem]{Corollary}

\theoremstyle{definition}

\newtheorem{example}[theorem]{Example}

\newtheorem{remark}[theorem]{Remark}

\begin{document}

\title[Jordan Derivations of some extension algebras]
{Jordan Derivations of some extension algebras}

\author{Yanbo Li and Feng Wei}

\address{Li: School of Mathematics and Statistics, Northeastern
University at Qinhuangdao, Qinhuangdao, 066004, P.R. China}

\email{liyanbo707@163.com}

\address{Wei: School of Mathematics, Beijing Institute of
Technology, Beijing, 100081, P. R. China}

\email{daoshuo@hotmail.com}

\begin{abstract}
In this paper, we mainly study Jordan derivations of dual extension
algebras and those of generalized one-point extension algebras. It
is shown that every Jordan derivation of dual extension algebras is
a derivation. As applications, we obtain that every Jordan
generalized derivation and every generalized Jordan derivation on
dual extension algebras are both generalized derivations. For
generalized one-point extension algebras, it is proved that under
certain conditions, each Jordan derivation of them is the sum of a
derivation and an anti-derivation.
\end{abstract}



\thanks{The work of the first author is supported by Fundamental Research
Funds for the Central Universities (N110423007).}

\maketitle

\section{Introduction}

Let us begin with some definitions. Let $\mathcal{R}$ be a
commutative ring with identity, $\mathcal{A}$ be a unital algebra
over $\mathcal{R}$ and $\mathcal{Z(A)}$ be the center of
$\mathcal{A}$. We denote the \textit{Jordan product} by $a\circ b=ab+ba$ for all
$a,b\in \mathcal{A}$. Recall that an $\mathcal{R}$-linear mapping
$\Theta$ from $\mathcal{A}$ into itself is called a
\textit{derivation} if
$$
\Theta(ab)=\Theta(a)b+a\Theta(b)
$$
for all $a, b\in \mathcal{A}$, an \textit{anti-derivation} if
$$\Theta (ab)=\Theta (b)a+b\Theta(a)$$ for all $a, b\in
\mathcal{A}$,
and a \textit{Jordan derivation} if
$$ \Theta(a\circ b)=\Theta(a)\circ b+a
\circ \Theta(b).$$ Every derivation is obviously a Jordan
derivation. The converse statement is in general not true. Moreover,
in the $2$-torsion free case the definition of a Jordan derivation
is equivalent to
$$ \Theta(x^2)=\Theta(x)x+x\Theta(x)
$$
for all $x\in \mathcal{A}$.  Those Jordan derivations which are
not derivations are said to be \textit{proper}.

\medskip

There has been an increasing interest in the study of Jordan
derivations of various algebras since last decades. The standard
problem is to find out whether a Jordan derivation degenerate to a
derivation. Jacobson and Rickart \cite{JR} proved that every Jordan
derivation of a full matrix algebra over a $2$-torsion free unital
ring is a derivation by relating the problem to the decomposition of
Jordan homomorphisms. In \cite{He}, Herstein showed that every
Jordan derivation from a $2$-torsion free prime ring into itself is
also a derivation. Zhang and Yu \cite{ZY} obtained that every Jordan
derivation on a triangular algebra with faithful assumption is a
derivation. This result was extended to the higher case by Xiao and
Wei \cite{XW1}.  They obtained that any Jordan higher derivation on
a triangular algebra is a higher derivation. The aforementioned
results have been extended to different rings and algebras in
various directions, see \cite{Br, Cu, Lu1, Lu2, ZY} and the
references therein.

Path algebras of quivers come up naturally in the study of tensor
algebras of bimodules over semisimple algebras. It is well known
that any finite dimensional basic $K$-algebra is given by a quiver
with relations when $K$ is an algebraically closed field. In
\cite{GL}, Guo and Li studied the Lie algebra of differential
operators on a path algebra $K\Gamma$ and related this Lie algebra
to the algebraic and combinatorial properties of the path algebra
$K\Gamma$. In \cite{LW2}, the current authors studied Lie
derivations and Jordan derivations of a class of path algebras of
quivers without oriented cycles, which can be viewed as one-point
extensions. It is proved in this case that each Lie derivation is of
the standard form and each Jordan derivation is a derivation.
Moreover, the standard decomposition of a Lie derivation is unique.

For path algebras of finite quivers without oriented cycles, Xi
\cite{X1} constructed their dual extension algebras to study
quasi-hereditary algebras. This construction were  further refined in
details in \cite{DX1, DX3, X2} by Deng and Xi. A more
general construction, the twisted doubles, were studied in
\cite{DX2, KX, X3}. It turns out that dual extension algebras inherit some nice
properties of the representation theory aspect from given
algebras. On the other hand, in \cite{LW3}, the current authors
proved that all Lie derivations of the dual extension algebra are
of the standard form. Then it is natural to ask whether all Jordan
derivations of dual extension algebras are derivations. We will
give a positive answer in this paper. More
precisely, one of the main results of this paper is as follows:\\

\noindent{\bf Theorem.} \,\,
{\em Let $K$ be a field with
${\rm char}\neq 2$. Let $(\Gamma, \rho)$ be a finite
quiver without oriented cycles. Then each Jordan derivation on the
dual extension algebra of path algebra $K(\Gamma,\rho)$ is a derivation.}
\\

It should be remarked that each associative algebra with non trivial
idempotents is isomorphic to a generalized matrix algebra. The form
of Jordan derivations on generalized matrix algebras has been
characterized by current authors in \cite{LWW}. We proved that under
certain conditions, each Jordan derivation is the sum of a
derivation and an anti-derivation. An example of proper Jordan
derivation was also given there. To find a proper Jordan derivation
is not an easy task in general. Recently Bencovi\v{c} \cite{B3}
introduced the so-called singular Jordan derivations which are
usually anti-derivations. He gave a sufficient condition for a
Jordan derivation on a unital algebra with a nontrivial idempotent
to be the sum of a derivation and a singular Jordan derivation. Our
result on Jordan derivations of dual extension algebras implies that
neither the conditions in \cite{LWW} nor those in \cite{B3} are
necessary. Of course, we want to give other examples to illustrate
this fact. The so-called generalized one-point extension algebras
introduced in \cite{LW3} just provide us another class of examples.
We prove that under certain conditions, each Jordan derivation on a
generalized one-point extension algebra is the sum of a derivation
and an anti-derivation.

\smallskip

The paper is organized as follows. After a quick review of some
needed preliminaries on path algebras and generalized matrix
algebras in Section 2, we investigate Jordan derivations of dual
extension algebras in Section 3. Jordan generalized derivations and
generalized Jordan derivations are also considered. Then in Section
4, we study Jordan derivations of generalized one-point extension
algebras. An interesting example will also be given there.

\bigskip

\section{Path algebras and generalized matrix algebras}

In this section, we give a quick review of path algebras of
quivers and generalized matrix algebras. For more details, we
refer the reader to \cite{ARS} and \cite{XW1}.

\subsection{Path algebras}
Recall that a \textit{finite quiver} $\Gamma=(\Gamma_0, \Gamma_1)$
is an oriented graph with the set of vertices $\Gamma_0$ and the
set of arrows between vertices $\Gamma_1$ being both finite. For
an arrow $\alpha$, we write $s(\alpha)=i$ and $e(\alpha)=j$ if it
is from the vertex $i$ to the vertex $j$. A \textit{sink} is a
vertex without arrows beginning at it and a \textit{source} is a
vertex without arrows ending at it. A \textit{nontrivial path} in
$\Gamma$ is an ordered sequence of arrows
$p=\alpha_n\cdots\alpha_1$ such that $e(\alpha_m)=s(\alpha_{m+1})$
for each $1\leq m<n$. Define $s(p)=s(\alpha_1)$ and
$e(p)=e(\alpha_n)$. The length of $p$ is defined to be $n$. A
\textit{trivial path} is the symbol $e_i$ for each $i\in
\Gamma_0$. In this case, we set $s(e_i)=e(e_i)=i$. The length of a
trivial path is defined to be zero. A nontrivial path $p$ is
called an \textit{oriented cycle} if $s(p)=e(p)$. Let us denote the set
of all paths by $\mathscr{P}$.

Let $K$ be a field and $\Gamma$ a quiver. Then the path algebra
$K\Gamma$ is the $K$-algebra generated by the paths in $\Gamma$
and the product of two paths $x=\alpha_n\cdots\alpha_1$ and
$y=\beta_t\cdots\beta_1$ is defined by
$$
xy=\left\{
\begin{array}{ll}
\alpha_n\cdots\alpha_1\beta_t\cdots\beta_1, & \mbox{$e(y)=s(x)$}\\
0, & \mbox{otherwise}.
\end{array}
\right.
$$
Clearly, $K\Gamma$ is an associative algebra with the identity
$1=\sum_{i\in \Gamma_0}e_i,$ where $e_i(i\in \Gamma_0)$ are
pairwise orthogonal primitive idempotents of $K\Gamma$.

A {\em relation} $\sigma$ on a quiver $\Gamma$ over a field $K$ is
a $K$-linear combination of paths $$\sigma=\sum_{i=1}^nk_ip_i,$$
where $k_i\in K$ and
$$e(p_1)=\cdots=e(p_n),\,\,\, s(p_1)=\cdots=s(p_n).$$ Moreover, the
number of arrows in each path is assumed to be at least 2. Let
$\rho$ be a set of relations on $\Gamma$ over $K$. The pair
$(\Gamma, \rho)$ is called a \textit{quiver} with relations over
$K$. Denote by $<\rho>$ the ideal of $K\Gamma$ generated by the
set of relations $\rho$. The $K$-algebra $K(\Gamma,
\rho)=K\Gamma/<\rho>$ is always associated with $(\Gamma, \rho)$.
For arbitrary element $x\in K\Gamma$, write by $\overline x$ the
corresponding element in $K(\Gamma, \rho)$. We often {\em write
$\overline x$ as $x$ if there is no confusion caused.}

\subsection{Generalized matrix algebras}

The definition of generalized matrix algebras is given by a Morita
context. Let $\mathcal{R}$ be a commutative ring with identity. A
Morita context consists of two $\mathcal{R}$-algebras $A$ and $B$,
two bimodules $_AM_B$ and $_BN_A$, and two bimodule homomorphisms
called the pairings $\Phi_{MN}: M\underset {B}{\otimes}
N\longrightarrow A$ and $\Psi_{NM}: N\underset {A}{\otimes}
M\longrightarrow B$ satisfying the following commutative diagrams:
$$
\xymatrix{ M \underset {B}{\otimes} N \underset{A}{\otimes} M
\ar[rr]^{\hspace{8pt}\Phi_{MN} \otimes I_M} \ar[dd]^{I_M \otimes
\Psi_{NM}} && A
\underset{A}{\otimes} M \ar[dd]^{\cong} \\  &&\\
M \underset{B}{\otimes} B \ar[rr]^{\hspace{10pt}\cong} && M }
{\rm and} \xymatrix{ N \underset {A}{\otimes} M
\underset{B}{\otimes} N \ar[rr]^{\hspace{8pt}\Psi_{NM}\otimes I_N}
\ar[dd]^{I_N\otimes \Phi_{MN}} && B
\underset{B}{\otimes} N \ar[dd]^{\cong}\\  &&\\
N \underset{A}{\otimes} A \ar[rr]^{\hspace{10pt}\cong} && N
\hspace{2pt}.}
$$
Let us write this Morita context as $(A, B, _AM_B, _BN_A,
\Phi_{MN}, \Psi_{NM})$. If $(A, B, _AM_B,$ $ _BN_A,$ $ \Phi_{MN},
\Psi_{NM})$ is a Morita context, then the set
$$
\left[
\begin{array}
[c]{cc}%
A & M\\
N & B\\
\end{array}
\right]=\left\{ \left[
\begin{array}
[c]{cc}%
a& m\\
n & b\\
\end{array}
\right] \vline a\in A, m\in M, n\in N, b\in B \right\}
$$
form an $\mathcal{R}$-algebra under matrix-like addition and
matrix-like multiplication. There is no constraint condition
concerning the bimodules $M$ and $N$. Of course, they can be equal
to zeros. Such an $\mathcal{R}$-algebra is called a
\textit{generalized matrix algebra} of order 2 and is usually
denoted by $\mathcal{G}=
\left[\smallmatrix A & M\\
N & B \endsmallmatrix \right]$ or $\mathcal{G}=(A, M, N, B)$.
The structure and properties of linear mappings on generalized matrix
algebras have been investigated in our systemic works \cite{LW1,
LW2, LWW, LX, XW1}.

Any unital $\mathcal{R}$-algebra $\mathcal {A}$ with nontrivial
idempotents is isomorphic to a generalized matrix algebra as
follows:
\begin{align*}
\mathcal{G} & = \left[
\begin{array}
[c]{cc}%
e\mathcal {A}e & e\mathcal {A}(1-e)\\
(1-e)\mathcal {A}e & (1-e)\mathcal {A}(1-e)\\
\end{array}
\right]\\ & =\left\{ \hspace{2pt} \left[
\begin{array}
[c]{cc}%
eae & ec(1-e)\\
(1-e)de & (1-e)b(1-e)\\
\end{array}
\right] \hspace{3pt} \vline \hspace{3pt} a, b, c, d\in \mathcal
{A}\hspace{3pt} \right\},
\end{align*}
where $e$ is a nontrivial idempotent in $\mathcal{A}$.

\bigskip\medskip

\section{Jordan Derivations of dual extension}

Lest us first recall the definition of dual extension algebras introduced in \cite{X1}. Let $\Lambda=K(\Gamma, \rho)$, where $\Gamma$ is a
finite quiver. Let $\Gamma^{\ast}$ to be a quiver whose vertex set
is $\Gamma_0$ and
$$\Gamma_1^{\ast}=\{\alpha^{\ast}: i\rightarrow j\mid \alpha: j\rightarrow i
\,\,\,\text{is an arrow in} \,\,\,\Gamma_1\}.$$ Let
$p=\alpha_n\cdots\alpha_1$ be a path in $\Gamma$. Write the path
$\alpha_1^{\ast}\cdots\alpha_n^{\ast}$ in $\Gamma^{\ast}$ by
$p^{\ast}$. Define $\mathscr{D}(\Lambda)$ to be the path algebra of
the quiver $(\Gamma_0, \Gamma_1\cup\Gamma_1^{\ast})$ with relations
$$\rho\,\,\cup\,\,\rho^{\ast}\,\,\cup\,\,\{\alpha\beta^{\ast}\mid\alpha,\beta\in\Gamma_1\}.$$
If $\Gamma$ has no oriented cycles, then $\mathscr{D}(\Lambda)$ is
called the {\em dual extension} of $\Lambda$. A more general
definition of dual extension algebras was given in \cite{X2}
to study global dimensions of dual extension algebras. We omit the details here because it
will not be involved in this paper. Clearly, if $|\Gamma_0|=1$,
then the algebra is trivial. Let us assume that $|\Gamma_0|\geq 2$
from now on. Then $\mathscr{D}(\Lambda)$ is isomorphic to a
generalized matrix
algebra $\mathcal{G}=\left[\smallmatrix A & M\\
N & B \endsmallmatrix \right]$. According to the construction of dual
extension, it is easy to verify that the pairings $\Phi_{MN}=0$ and
$\Psi_{NM}\neq 0$. If $M\neq 0$, then $N\neq 0$. Moreover, it is
helpful to point out that $M$ need not to be faithful as left
$A$-module or as right $B$-module. Some examples were given in
\cite{LW3}.

We always assume, without specially mentioned, that
every algebra and every bimodule considered is 2-torsion free. Let
us recall some indispensable descriptions about derivations and
Jordan derivations of generalized matrix algebras. For
details, we refer the reader to \cite{LW1} and \cite{LWW}.

\begin{lemma}\cite[Proposition 4.2]{LW1}\label{3.1}
An additive map $\Theta$ from $\mathcal{G}$ into itself is a
derivation if and only if it has the form
$$
\begin{aligned}
& \Theta\left(\left[
\begin{array}
[c]{cc}%
a & m\\
n & b\\
\end{array}
\right]\right)=& \left[
\begin{array}
[c]{cc}%
\delta_1(a)-mn_0-m_0n & am_0-m_0b+\tau_2(m)\\
n_0a-bn_0+\nu_3(n) & n_0m+nm_0+\mu_4(b)\\
\end{array}
\right] ,(\bigstar 1) \\
& \forall \left[
\begin{array}
[c]{cc}%
a & m\\
n & b\\
\end{array}
\right]\in \mathcal{G},
\end{aligned}
$$
where $m_0\in M, n_0\in N$ and
$$
\begin{aligned} \delta_1:& A \longrightarrow A, &
 \tau_2: & M\longrightarrow M, &  \nu_3: & N\longrightarrow N , &
\mu_4: & B\longrightarrow B
\end{aligned}
$$
are all $\mathcal{R}$-linear mappings satisfying the following
conditions:
\begin{enumerate}
\item[(1)] $\delta_1$ is a derivation of $A$ with
$\delta_1(mn)=\tau_2(m)n+m\nu_3(n);$

\item[(2)] $\mu_4$ is a derivation of $B$ with
$\mu_4(nm)=n\tau_2(m)+\nu_3(n)m;$

\item[(3)] $\tau_2(am)=a\tau_{2}(m)+\delta_1(a)m$ and
$\tau_2(mb)=\tau_2(m)b+m\mu_4(b);$

\item[(4)] $\nu_3(na)=\nu_3(n)a+n\delta_1(a)$ and
$\nu_3(bn)=b\nu_3(n)+\mu_4(b)n.$
\end{enumerate}
\end{lemma}

\begin{lemma}\cite[Proposition 4.2]{LWW}\label{3.2}
An additive map $\Theta$ from $\mathcal{G}$ into itself is a
Jordan derivation if and only if it is of the form
$$
\begin{aligned}
& \Theta\left(\left[
\begin{array}
[c]{cc}%
a & m\\
n & b\\
\end{array}
\right]\right)\\ =& \left[
\begin{array}
[c]{cc}%
\delta_1(a)-mn_0-m_0n & am_0-m_0b+\tau_2(m)+\tau_3(n)\\
n_0a-bn_0+\nu_2(m)+\nu_3(n) & n_0m+nm_0+\mu_4(b)\\
\end{array}
\right] ,(\bigstar 2)\\
& \forall \left[
\begin{array}
[c]{cc}%
a & m\\
n & b\\
\end{array}
\right]\in \mathcal{G},
\end{aligned}
$$
where $m_0\in M, n_0\in N$ and
$$
\begin{aligned} \delta_1:& A \longrightarrow A, &  \tau_2:
& M\longrightarrow M, & \tau_3: & N\longrightarrow M,\\
\nu_2: & M\longrightarrow N, & \nu_3: & N\longrightarrow N, & \mu_4:
& B\longrightarrow B
\end{aligned}
$$
are all $\mathcal {R}$-linear mappings satisfying the following
conditions:
\begin{enumerate}
\item[(1)] $\delta_1$ is a Jordan derivation on $A$ and
$\delta_1(mn)=\tau_2(m)n+m\nu_3(n);$

\item [(2)] $\mu_4$ is a Jordan derivation on $B$ and
$\mu_4(nm)=n\tau_2(m)+\nu_3(n)m;$

\item[(3)] $\tau_2(am)=a\tau_2(m)+\delta_1(a)m$ and
$\tau_2(mb)=\tau_2(m)b+m\mu_4(b);$

\item[(4)] $\nu_3(bn)=b\nu_3(n)+\mu_4(b)n$ and
$\nu_3(na)=\nu_3(n)a+n\delta_1(a);$

\item[(5)] $\tau_3(na)=a\tau_3(n)$, $\tau_3(bn)=\tau_3(n)b$, $n\tau_3(n)=0$,
$\tau_3(n)n=0;$

\item[(6)] $\nu_2(am)=\nu_2(m)a$, $\nu_2(mb)=b\nu_2(m)$, $m\nu_2(m)=0$,
$\nu_2(m)m=0.$
\end{enumerate}
\end{lemma}

Let $\Lambda=K(\Gamma, \rho)$, where $\Gamma$ is a finite
connected quiver without oriented cycles, and let
$\mathscr{D}(\Lambda)$ be the dual extension algebra. Assume that
$i\in \Gamma_0$ is a source
and $\mathscr{D}(\Lambda)\simeq \mathcal{G}=\left[\smallmatrix A & M\\
N & B \endsmallmatrix \right]$, where $B\simeq
e_i\mathscr{D}(\Lambda)e_i$. Thus an arbitrary Jordan derivation on
$\mathscr{D}(\Lambda)$ can be characterized by the methods of
generalized matrix algebras as follows.

\begin{lemma}\label{3.3}
Let $\Theta$ be a Jordan derivation  of $\mathscr{D}(\Lambda)$.
Then $\Theta$ is of the form
$$
\begin{aligned}
& \Theta\left(\left[
\begin{array}
[c]{cc}%
a & m\\
n & b\\
\end{array}
\right]\right) =& \left[
\begin{array}
[c]{cc}%
\delta_1(a) & am_0-m_0b+\tau_2(m)\\
n_0a-bn_0+\nu_3(n) & n_0m+nm_0+\mu_4(b)\\
\end{array}
\right] , (\bigstar 3)\\
& \forall \left[
\begin{array}
[c]{cc}%
a & m\\
n & b\\
\end{array}
\right]\in \mathcal{G},
\end{aligned}
$$
where $m_0\in M, n_0\in N$ and
$$
\begin{aligned} \delta_1:& A \longrightarrow A, &  \tau_2:
& M\longrightarrow M, & \nu_3: & N\longrightarrow N, & \mu_4: &
B\longrightarrow B
\end{aligned}
$$
are all $\mathcal {R}$-linear mappings satisfying the following
conditions:
\begin{enumerate}
\item[(1)] $\delta_1$ is a Jordan derivation on $A$;

\item [(2)] $\mu_4$ is a derivation on $B$ and
$\mu_4(nm)=n\tau_2(m)+\nu_3(n)m;$

\item[(3)] $\tau_2(am)=a\tau_2(m)+\delta_1(a)m$ and
$\tau_2(mb)=\tau_2(m)b+m\mu_4(b);$

\item[(4)] $\nu_3(bn)=b\nu_3(n)+\mu_4(b)n$ and
$\nu_3(na)=\nu_3(n)a+n\delta_1(a).$
\end{enumerate}
\end{lemma}

\begin{proof}
Let $\Theta$ be a Jordan derivation of $\mathscr{D}(\Lambda)$ with
the form $(\bigstar 2)$. We first prove that $\tau_3=0$,
$\nu_2=0$. Let $\alpha\in N$ be an arbitrary arrow. Then
$e(\alpha)=i$. Assume that $s(\alpha)=j$, where $j\in \Gamma_0$.
In view of condition (5) of Lemma \ref{3.2} we know that
$$
\tau_3(\alpha)=\tau_3(\alpha
e_j)=e_j\tau_3(\alpha).
$$
This implies that if $\tau_3(\alpha)\neq
0$, then $\alpha\tau_3(\alpha)\neq 0$. However,
$\alpha\tau_3(\alpha)\neq 0$ is impossible by condition (5) of
Lemma \ref{3.2}. Thus $\tau_3(\alpha)=0$ for all $\alpha\in N$.
Note that all path $p\in N$ with length more than $1$ is of the
form $\alpha p'$, where $\alpha$ is an arrow ending at $i$. Then
$\tau_3(p)=\tau_3(\alpha p')=p'\tau_3(\alpha)=0$. This completes
the proof of $\tau_3=0$. It is proved similarly that $\nu_2=0$.
Furthermore, we have from $B$ is commutative that every Jordan
derivation of $B$ is a derivation. Finally, the fact $\Phi_{MN}=0$ leads to
$mn_0=m_0n=0$ for all $m\in M$ and $n\in N$.
\end{proof}

Now we are in a position to describe Jordan derivations of a dual
extension algebra.

\begin{theorem}\label{3.4}
Let $\Gamma$ be a finite connected quiver without oriented cycles
and $\Lambda=K(\Gamma, \rho)$. Let $\mathscr{D}(\Lambda)$ be the
dual extension algebra of $\Lambda$. Then each Jordan derivation
of $\mathscr{D}(\Lambda)$ is a derivation.
\end{theorem}

\begin{proof}
If the algebra $\mathscr{D}(\Lambda)$ is trivial, then the theorem
clearly holds. Suppose that $\Gamma_0\geq 2$ and that $i\in
\Gamma_0$ is a source. Let $\Theta$ be a Jordan derivation on
$\mathscr{D}(\Lambda)$. Let us denote by $(\Gamma', \rho')$ the
quiver obtained by removing the vertex $i$ and the relations
starting at $i$ and write $\Lambda'=K(\Gamma', \rho')$. It follows
from Lemma \ref{3.3} that each Jordan derivation on
$\mathscr{D}(\Lambda)$ is a derivation if each Jordan derivation on
$\mathscr{D}(\Lambda')$ is a derivation. Thus it is sufficient to
determine whether every Jordan derivation on $\mathscr{D}(\Lambda')$
is a derivation. We continuously repeat this process and ultimately
arrive at the algebra $K$ after finite times, since $\Gamma_0$ is a
finite set. Clearly, every Jordan derivation on $K$ is a derivation.
This completes the proof.
\end{proof}

\begin{remark}
In \cite{LWW}, the current authors and L. W. Wyk proved that for a
generalized matrix algebra $\mathcal{G}=(A, M, N, B)$ with $M$
being faithful as left $A$-module and as right $B$-module, if the
pairings $\Phi_{MN}=0$ and $\Psi_{NM}=0$, then each Jordan
derivation of $\mathcal{G}$ is the sum of a derivation and an
anti-derivation. Theorem \ref{3.4} implies that neither the
faithful condition nor the pairings being both zero is necessary.
Benkovi\v{c} and \v{S}irovnik \cite{B3} introduced
the so-called singular Jordan derivations and showed that under
certain condition, each Jordan derivation of a generalized matrix
algebra is the sum of a derivation and a singular Jordan
derivation. Theorem \ref{3.4} also implies that Bencovi\v{c}'s
condition is not necessary.
\end{remark}

At the end of this section, let us characterize Jordan generalized
derivations and generalized Jordan derivations of dual extension
algebras. Let $\mathcal{R}$ be a commutative ring with identity,
$\mathcal{A}$ a unital algebra over $\mathcal{R}$ and
$\mathcal{M}$ an $(\mathcal{A}, \mathcal{A})$-bimodule. Recall that a
linear mapping $f\colon \mathcal{A}\rightarrow \mathcal{A}$ is called a
\emph{Jordan generalized derivation} if there exists a linear mapping
$d\colon \mathcal{A}\rightarrow\mathcal{A}$ such that
$$f(x\circ y)=f(x)\circ y+x\circ d(y)$$
for all $x,y\in\mathcal{A}$,where $d$ is called an associated
linear mapping of $f$. A linear mapping $f\colon \mathcal{A}\longrightarrow
\mathcal{A}$ is called a \emph{generalized Jordan derivation} if
there exists a linear mapping $d\colon \mathcal{A}\longrightarrow
\mathcal{A}$ such that $f(x\circ y)=f(x)y+f(y)x+xd(y)+yd(x)$ for
all $x, y\in \mathcal{A}$. A linear mapping $f
:\mathcal{A}\longrightarrow \mathcal{A}$ is called a
\emph{generalized derivation} if there exists a linear mapping $d\colon \mathcal{A}\longrightarrow \mathcal{A}$ such that
$f(xy)=f(x)y+xd(y)$ for all $x, y\in \mathcal{A}$.

For generalized Jordan derivations on dual extension algebras, the
following result is a simple corollary of \cite[Lemma 4.1]{B2} and
Theorem \ref{3.4}.

\begin{corollary}\label{3.6}
Every generalized Jordan derivation on a dual extension algebra
$\mathscr{D}(\Lambda)$ is a generalized derivation.
\end{corollary}

In order to deal with Jordan generalized derivations, we need the
following two lemmas obtained in \cite{LB} by the first author and
Benkovi\v{c}.

\begin{lemma}\cite[Proposition 2.1]{LB}\label{3.7}
Let $f:\mathcal{A\rightarrow M}$ be a generalized derivation with
an associated linear mapping $d$. Then $d$ is a derivation and
$f(x) =f(1) x+d(x) $ for
all $x\in\mathcal{A}.$
\end{lemma}

\begin{lemma}\cite[Theorem 2.2]{LB}\label{3.8}
Let $f:\mathcal{A\rightarrow M}$ be a Jordan generalized
derivation with an associated linear mapping $d$. The following
statements are equivalent:
\begin{enumerate}
\item[(1)] Element $f(1)$ belongs to the center of
$\mathcal{M}$.

\item[(2)] The mapping $d$ is a Jordan derivation and
$f(x)=f(1)x+d(x)$ for all $x\in\mathcal{A}$.
\end{enumerate}
\end{lemma}

Let $f$ be a Jordan generalized derivation on a dual extension
algebra $\mathscr{D}(\Lambda)$. It follows from Lemma \ref{3.7} ,
Lemma \ref{3.8} and Theorem \ref{3.4} that if $f(1)\in
\mathcal{Z}(\mathscr{D}(\Lambda))$, then every Jordan generalized
derivation is a generalized derivation. In order to prove that
$f(1)\in \mathcal{Z}(\mathscr{D}(\Lambda))$, the following lemma
obtained in \cite{LB} will play an important role.

\begin{lemma}\cite[Lemma 2.4]{LB}\label{3.9}
Let $f:\mathcal{A}\rightarrow\mathcal{M}$ be a Jordan generalized
derivation with an associated linear mapping $d\colon \mathcal{A}\rightarrow\mathcal{M}%
$. Then the following holds:
\begin{enumerate}
\item[(1)] $[ [ x, y] ,f(1)]  =0$ for all $x,y\in\mathcal{A}$;

\item[(2)] $f(1)=ef(1)e+(1-e) f(1)(1-e)$ for any
idempotent $e\in\mathcal{A}$.
\end{enumerate}
\end{lemma}

\begin{lemma}\label{3.10}
Let $f$ be a Jordan generalized derivation on a dual extension
algebra $\mathscr{D}(\Lambda)$. Then $f(1)\in
\mathcal{Z}(\mathscr{D}(\Lambda))$.
\end{lemma}

\begin{proof}
It is not difficult to see that $e_i(1-e_i)=(1-e_i)e_i=0$ for all $i\in \Gamma_0$. Applying Lemma \ref{3.9} (2) yields that
$$
e_if(1)=e_if(1)e_i=f(1)e_i.\eqno(3.1)
$$

Moreover, let $p$ be a nontrivial path with $s(p)=r$, $e(p)=t$ and
$s\neq t$. Then $p=[p, e_r]$. Combining this fact with Lemma
\ref{3.9} (2) gives that
$$
f(1)p=pf(1).\eqno(3.2)
$$

Furthermore, we claim that $e_if(1)e_j=0$ for arbitrary $i, j\in
\Gamma_0$ with $i\neq j$. In fact, Lemma \ref{3.9} (2) implies
that
$$f(1)=(e_i+e_j)f(1)(e_i+e_j)+(1-e_i-e_j)f(1)(1-e_i-e_j).\eqno(3.3)$$
Left multiplication of $(3.3)$ by $e_i$ leads to
$$
e_if(1)=e_if(1)e_i+e_if(1)e_j.\eqno(3.4)
$$
Note that
$e_if(1)=e_if(1)e_i$. Then (3.4) forces $e_if(1)e_j=0$. This
shows that the paths $p$ with $s(p)\neq e(p)$ do not appear in
the expansion of $f(1)$. Hence for all nontrivial paths $p$ with
$s(p)=e(p)$,
$$
f(1)p=pf(1).\eqno(3.5)
$$
Combining (3.1), (3.2) with
(3.5) yields that $f(1)\in \mathcal{Z}(\mathscr{D}(\Lambda))$.
\end{proof}

Applying Theorem \ref{3.4}, Lemma \ref{3.7} and Lemma \ref{3.10} yields
the following result.

\begin{proposition}
Every Jordan generalized derivation on a dual extension algebra
$\mathscr{D}(\Lambda)$ is a generalized derivation.
\end{proposition}

\bigskip

\section{Jordan derivations of generalized one-point extensions}

We introduced the notion of generalized one-point extension algebras
in \cite{LW3} and studied Lie
derivations on them. In this section, we will investigate Jordan
derivations of generalized one-point algebras. Let us first recall the definition.

Let $\Gamma=(\Gamma_0, \Gamma_1)$ be a finite quiver without
oriented cycles and $|\Gamma_0|\geq 2$. Let $\Gamma^{\ast}$ be a
quiver whose vertex set is $\Gamma_0$ and
$$\Gamma_1^{\ast}=\{\alpha^{\ast}: i\rightarrow j\mid \alpha: j\rightarrow i
\,\,\,\text{is an arrow in} \,\,\,\Gamma_1\}.$$ For a path
$p=\alpha_n\cdots\alpha_1$ in $\Gamma$, write the path
$\alpha_1^{\ast}\cdots\alpha_n^{\ast}$ in $\Gamma^{\ast}$ by
$p^{\ast}$. Given a set $\rho$ of relations, denote by
$\Lambda=K(\Gamma, \rho)$. Define the generalized one-point
extension algebra $E(\Lambda)$ to be the path algebra of the
quiver $(\Gamma_0, \Gamma_1\cup\Gamma_1^{\ast})$ with relations
$$\rho\,\,\cup\,\,\rho^{\ast}\,\,
\cup\,\,\{\alpha\beta^{\ast}\mid\alpha,\beta\in\Gamma_1\}\,\,
\cup\,\,\{\alpha^{\ast}\beta\mid\alpha,\beta\in\Gamma_1\}.$$

It is helpful to point out that if we choose a suitable
idempotent, then neither $M$ nor $N$ need not to be faithful. Let
us illustrate an example here.

\begin{example}\label{xxsec4.1}
Let $\Gamma$ be a quiver as follows
$$
\xymatrix@C=20mm{
  \bullet
  \ar@<0pt>[r]_(0){1}^{\alpha}  &
  \bullet
  \ar@<0pt>[r]_(0){2}^{\beta}  &
  \bullet &
  \bullet
  \ar@<0pt>@[r][l]_(0.5){\gamma}^(0){3}^(1){4} }
$$
and let $\Lambda=K\Gamma$. The generalized one-point extension
algebra $E(\Lambda)$ has a basis $$\{e_1, e_2, e_3, e_4, \alpha,
\beta, \gamma, \alpha^{\ast}, \beta^{\ast}, \gamma^{\ast},
 \beta\alpha, \alpha^{\ast}\beta^{\ast}\}.$$

Taking the nontrivial idempotent to be $e_1+e_2$, then
$E(\Lambda)$ is isomorphic to a generalized matrix algebra
$\mathcal {G}=\left[
\begin{array}
[c]{cc}%
A & M\\
N & B\\
\end{array}
\right]$, where $A$ has a basis $\{e_1,\, e_2,\, \alpha,\,
\alpha^{\ast}\}$, $B$ has a basis $\{e_3,\, e_4,\, \gamma,\,
\gamma^{\ast}\}$, $M$ has a basis $\{\alpha^{\ast}\beta^{\ast},\,
\beta^{\ast}\}$ and $N$ has a basis $\{\beta,\, \beta\alpha\}$. It
is easy to check that $\alpha\in {\rm Ann}(_AM)$ and $\gamma\in
{\rm Ann}(M_B)$, that is, $M$ is neither faithful as left
$A$-module nor as right $B$-module. Similarly, we obtain
$\gamma\in {\rm Ann}(_BN)$ and $\alpha\in {\rm Ann}(N_A)$ that
is, $N$ is neither faithful as left $B$-module nor as right
$A$-module.

Moreover, in \cite{B3} Benkovi\v{c} proved that for a generalized
matrix algebra $\mathcal{G}=
\left[\smallmatrix A & M\\
N & B \endsmallmatrix \right]$, if
\begin{enumerate}
\item[(1)] $aM=0$ and $Na=0$ imply that $a=0$;

\item[(2)] $Mb=0$ and $bN=0$ imply that $b=0$,
\end{enumerate}
then every Jordan derivation on $\mathcal{G}$ is the sum of a
derivation and an anti-derivation. Clearly, our example does not
satisfy Benkovi\v{c}'s conditions.
\end{example}

The aim of this section is to prove that under certain conditions,
each Jordan derivation of $E(\Lambda)$ is the sum of a derivation
and an anti-derivation. Let us first characterize anti-derivations of
generalized one-point extension algebras.

\begin{lemma}\label{xxsec4.2}
Let $\Gamma$ be a finite quiver without oriented cycles and
$\Lambda=K(\Gamma, \rho)$. Let $\Theta$ be an anti-derivation on
$E(\Lambda)$ and $\alpha\in \Gamma_1\cup\Gamma_1^{\ast}$ with
$s(\alpha)=r$ and $e(\alpha)=t$. Then
\begin{enumerate}
\item[(1)] $\Theta(e_i)=\sum\limits_{s(p)=i,\,\, or\,\,
e(p)=i}k_p^ip$;

\medskip

\item[(2)] $\Theta(\alpha)=\sum\limits_{s(p)=t,
e(p)=r}k_p^{\alpha}p.$
\end{enumerate}

Moreover, $\Theta(p)=0$ for all path $p$ with length more than
one. If there exists a nontrivial path $\beta$ such that
$\beta\alpha\neq 0$ or $\alpha\beta\neq 0$, then
$\Theta(\alpha)=0$.
\end{lemma}

\begin{proof}
(1) Suppose that
$$
\Theta(e_r)=\sum_{i\in
\Gamma_0}k_ie_i+\sum_{s(p)\neq e(p)}k_p^rp.\eqno(4.1)
$$
It follows
from the fact $e_r^2=e_i$ that
$$
\Theta(e_r)=\Theta(e_r)e_r+e_r\Theta(e_r).\eqno(4.2)
$$
Combining (4.1) with (4.2) gives that $k_r=0$. If there exists
$j\in \Gamma_0$ with $i\neq j$ such that $k_j\neq 0$, then the
coefficient of $e_j$ in the expansion of $\Theta(e_r)e_j$ is
$k_j$. On the other hand, since $e_j$ does not appear in the
expansion of $\Theta(e_j)$, we conclude that $e_j$ does not appear in
the expansion of $e_r\Theta(e_j)$ too. This implies that
$\Theta(e_je_r)\neq 0$, which is impossible.

\smallskip

(2) Let $\Theta$ be an anti-derivation on $E(\Lambda)$ and let
$\alpha\in \Gamma_1$ with $s(\alpha)=r$ and $e(\alpha)=t$. Suppose that
$$
\Theta(\alpha)=\sum_{i\in \Gamma_0}k_ie_i+\sum_{s(p)\neq
e(p)}k_p^{\alpha}p. \eqno(4.3)
$$
Then on one hand,
$$
\Theta(\alpha)=\Theta(e_t\alpha)=\Theta(\alpha)e_t+\alpha\Theta(e_t). \eqno(4.4)
$$
Substituting (4.3) into (4.4) gives that
\begin{eqnarray*}
& &\sum_{i\in \Gamma_0}k_ie_i+\sum_{s(p)\neq e(p)}k_p^{\alpha}p\\
&=&(\sum_{i\in \Gamma_0}k_ie_i+\sum_{s(p)\neq
e(p)}k_p^{\alpha}p)e_t+\alpha\Theta(e_t)\\
&=&k_te_t+\sum_{s(p)=t}k_p^{\alpha}p+\alpha\Theta(e_t).
\end{eqnarray*}
On the other hand,
$$
\Theta(\alpha)=\Theta(\alpha e_r)=e_r\Theta(\alpha)+\Theta(e_r)\alpha. \eqno(4.5)
$$
Substituting (4.3) into (4.5) yields that
\begin{eqnarray*}
& &\sum_{i\in \Gamma_0}k_ie_i+\sum_{s(p)\neq e(p)}k_p^{\alpha}p\\
&=&e_r(\sum_{i\in \Gamma_0}k_ie_i+\sum_{s(p)\neq
e(p)}k_p^{\alpha}p)+\Theta(e_r)\alpha\\
&=& k_re_r+\sum_{e(p)=r}k_p^{\alpha}p+\Theta(e_r)\alpha.
\end{eqnarray*}
By the above equalities it follows that
\begin{eqnarray*}
\sum_{i\in \Gamma_0}k_ie_i+\sum_{s(p)\neq
e(p)}k_p^{\alpha}p&=&k_te_t+\sum_{s(p)=t}k_p^{\alpha}p+\alpha\Theta(e_t)\\
&=&k_re_r+\sum_{e(p)=r}k_p^{\alpha}p+\Theta(e_r)\alpha.
\end{eqnarray*}
This implies that $k_i=0$ for all $i\in \Gamma_0$ and the
coefficients of all paths $p$ with $s(p)\neq t$ or $e(p)\neq r$ in
the expansion of $\Theta(\alpha)$ are zero, that is,
$$
\Theta(\alpha)=\sum_{s(p)=t, e(p)=r}k_p^{\alpha}p.\eqno(4.6)
$$
If there exists a non trivial path $\beta$ such that
$\beta\alpha\neq 0$, then $\alpha\beta^{\ast}=0$. However,
$\Theta(\alpha\beta^{\ast})=\Theta(\beta^{\ast})\alpha+\beta^{\ast}\Theta(\alpha)$.
If $\Theta(\alpha)\neq 0$, then by (4.5) we know that
$$
\beta^{\ast}\Theta(\alpha)=\sum_{s(p)=t, e(p)=r}k_p^{\alpha}\beta
p\neq 0,
$$
and hence $\Theta(\alpha\beta^{\ast})\neq 0$, which is a
contradiction. This forces that $\Theta(\alpha)=0$. Similarly, we can
show that if $\alpha\beta\neq 0$, then $\Theta(\alpha)=0$.
\end{proof}

Since $\Gamma$ is a quiver without oriented cycles, we can take a
source $i$ in $\Gamma$. Let $e_i$ be the corresponding idempotent
in $E(\Lambda)$. Then $E(\Lambda)$ is isomorphic to a generalized
matrix algebra $\mathcal{G}=
\left[\smallmatrix A & M\\
N & B \endsmallmatrix \right]$ with $A\simeq
E(\Lambda')$, where the quiver $\Gamma'$ of $\Lambda'$ is obtained
by removing the vertex $i$ and the relations starting at $i$.
Moreover, we have from the construction of $E(\Lambda)$ that the
bilinear pairings are both zero. In this case, the form $(\bigstar
2)$ of any Jordan derivation of $E(\Lambda)$ becomes as follows:

\begin{lemma}\label{xxsec4.3}
Let $\Lambda=K\Gamma$ and $E(\Lambda)$ be the generalized one-point
extension. Then an arbitrary Jordan derivation $\Theta$ on $E(\Lambda)$ is
of the form
$$
\begin{aligned}
& \Theta\left(\left[
\begin{array}
[c]{cc}%
a & m\\
n & b\\
\end{array}
\right]\right)\\ =& \left[
\begin{array}
[c]{cc}%
\delta_1(a) & am_0-m_0b+\tau_2(m)+\tau_3(n)\\
n_0a-bn_0+\nu_2(m)+\nu_3(n) & 0\\
\end{array}
\right] ,(\blacklozenge 1)\\
& \forall \left[
\begin{array}
[c]{cc}%
a & m\\
n & b\\
\end{array}
\right]\in \mathcal{G},
\end{aligned}
$$
where $m_0\in M, n_0\in N$ and
$$
\begin{aligned} \delta_1:& A \longrightarrow A, &  \tau_2:
& M\longrightarrow M, & \tau_3: & N\longrightarrow M, & \nu_2: &
M\longrightarrow N, & \nu_3: & N\longrightarrow N,
\end{aligned}
$$
are all $\mathcal {R}$-linear mappsing satisfying the following
conditions:
\begin{enumerate}
\item[(1)] $\delta_1$ is a Jordan derivation on $A$.

\item[(2)] $\tau_2(am)=a\tau_2(m)+\delta_1(a)m$ and
$\tau_2(mb)=\tau_2(m)b;$

\item[(3)] $\nu_3(bn)=b\nu_3(n)$ and
$\nu_3(na)=\nu_3(n)a+n\delta_1(a);$

\item[(4)] $\tau_3(na)=a\tau_3(n)$, $\tau_3(bn)=\tau_3(n)b$,

\item[(5)] $\nu_2(am)=\nu_2(m)a$, $\nu_2(mb)=b\nu_2(m)$
\end{enumerate}
\end{lemma}

\begin{proof}
We only need to prove $\mu_4=0$. But this is clear because $\mu_4$
is a Jordan derivation on $B=K$.
\end{proof}

In \cite{LWW}, the form of an arbitrary anti-derivation on a generalized
matrix algebra $\mathcal{G}=
\left[\smallmatrix A & M\\
N & B \endsmallmatrix \right]$ has been characterized
under the condition that $M$ being faithful as left $A$-module and
also as right $B$-module. If we remove the faithful assumption on $M$,
the form of an anti-derivation on $\mathcal{G}$ is as follows:

\begin{lemma}\label{xxsec4.4}
An additive mapping $\Theta$ is an anti-derivation of $\mathcal{G}$ if
and only if $\Theta$ has the form
$$
\begin{aligned}
& \Theta\left(\left[
\begin{array}
[c]{cc}%
a & m\\
n & b\\
\end{array}
\right]\right)  \\=& \left[
\begin{array}
[c]{cc}%
\delta_1(a) & am_0-m_0b+\tau_3(n)\\
n_0a-bn_0+\nu_2(m) & \mu_4(b)\\
\end{array}
\right] , (\blacklozenge 2) \\& \forall \left[
\begin{array}
[c]{cc}%
a & m\\
n & b\\
\end{array}
\right]\in \mathcal{G},
\end{aligned}
$$
where $m_0\in M, n_0\in N$ satisfying for all $a, a'\in A$, $b,
b'\in B$, $m\in M$ and $n\in N$
\begin{enumerate}
\item[{\rm(1)}] $[a, a']m_0=0$, $m_0[b, b']=0$, $n_0[a, a']=0$,
$[b, b']n_0=0;$ \item[{\rm(2)}] $m_0n=0$, $nm_0=0$, $mn_0=0$,
$n_0m=0;$
\end{enumerate} and
$$
\begin{aligned} \delta_1: & A\longrightarrow A, & \tau_3: & N\longrightarrow M, &
\nu_2: & M\longrightarrow N, & \mu_4: & B\longrightarrow B
\end{aligned}
$$
are $\mathcal{R}$-linear mappings satisfying for all $a\in A$, $b\in
B$, $m, m'\in M$ and $n, n'\in N$
\begin{enumerate}
\item[{\rm(3)}] $\delta_1$ is an anti-derivation on $A$ and
$\delta_1(mn)=0,$ $\delta_1(a)m=0$, $n\delta_1(a)=0;$
\item[{\rm(4)}] $\mu_4$ is an anti-derivation on $B$ and
$\mu_4(nm)=0,$ $m\mu_4(b)=0,$ $\mu_4(b)n=0;$ \item[{\rm(5)}]
$\tau_3(na)=a\tau_3(n)$, $\tau_3(bn)=\tau_3(n)b$, $n\tau_3(n')=0$,
$\tau_3(n)n'=0;$ \item[{\rm(6)}] $\nu_2(am)=\nu_2(m)a$,
$\nu_2(mb)=b\nu_2(m)$, $m\nu_2(m')=0$, $\nu_2(m)m'=0$.
\end{enumerate}
\end{lemma}

\begin{proof}
It can be proved similarly as that of \cite[Proposition 3.6]{LWW}.
\end{proof}

As a consequence of Lemma \ref{xxsec4.3} and Lemma \ref{xxsec4.4} we have

\begin{proposition}\label{xxsec4.5}
Let $\Theta$ be a Jordan derivation on a generalized one-point
extension algebra $E(\Lambda)\simeq \mathcal{G}=
\left[\smallmatrix A & M\\
N & B \endsmallmatrix \right]$. If
there exists an anti-derivation $f$ on $A$ with ${\rm
Im}(f)\subset {\rm Ann}(_AM)$ such that $\delta_1-f$ is a
derivation of $A$, then $\Theta$ is the sum of a derivation and an
anti-derivation.
\end{proposition}

We are now in a position to state the main result of this section.

\begin{theorem}\label{xxsec4.6}
Let $\Gamma$ be a finite quiver without oriented cycles and
$\Lambda=K(\Gamma, \rho)$. If there is no path $p$ with length
more than one, then every Jordan derivation on the generalized one
point extension algebra $E(\Lambda)$ is the sum of a derivation
and an anti-derivation.
\end{theorem}

\begin{proof}
Let $\Theta$ be a Jordan derivation on $E(\Lambda)$. Then by Lemma
\ref{xxsec4.2} it is of the form $(\blacklozenge 1)$. We claim that
if each Jordan derivation on $A$ is the sum of a derivation and an
anti-derivation, then so is $E(\Lambda)$. In fact, assume that
$\delta_1=d+f$, where $d$ is a derivation of $A$ and $f$ is an
anti-derivation of $A$. By Lemma \ref{xxsec4.2} we know that all
$e_i$ do not appear in $f(a)$ for $a\in A$. Note that the length of
each path is not more than one. This implies that $f(a)m=0$ for all
$a\in A$ and $m\in M$. Similarly, we can show that $nf(a)=0$ for all
$a\in A$ and $n\in N$. Define a linear mapping $f'$ on $E(\Lambda)$
by
$$
\begin{aligned}
f'\left(\left[
\begin{array}
[c]{cc}%
a & m\\
n & b\\
\end{array}
\right]\right)= \left[
\begin{array}
[c]{cc}%
f(a) & \tau_3(n)\\
\nu_2(m) & 0\\
\end{array}
\right], \hspace{8pt} \forall \left[
\begin{array}
[c]{cc}%
a & m\\
n & b\\
\end{array}
\right]\in E(\Lambda)
\end{aligned}
$$
Then Lemma \ref{xxsec4.2} and Lemma \ref{xxsec4.4} give that $f'$ is
a anti-derivation of $E(\Lambda)$. Furthermore, the linear mapping
$$
\begin{aligned}
D\left(\left[
\begin{array}
[c]{cc}%
a & m\\
n & b\\
\end{array}
\right]\right)= \left[
\begin{array}
[c]{cc}%
d(a) & am_0-m_0b+\tau_2(m)\\
n_0a-bn_0+\nu_3(n) & 0\\
\end{array}
\right]
\end{aligned}
$$
is a derivation of $E(\Lambda)$. This completes
the proof of our claim. Repeating this process, we arrive at the
algebra $K$, on which every Jordan derivation is zero. This
completes the proof.
\end{proof}

Finally, we illustrate an example which satisfies the condition of
Theorem \ref{xxsec4.6}.

\begin{example}
Let $\Gamma$ be a quiver as follows
$$\xymatrix@C=25mm{
  \bullet
  \ar@<0pt>[r]_(0){1}^{\alpha}  &
  \bullet &
  \bullet
  \ar@<0pt>@[r][l]_(0.5){\beta}^(0){3}^(1){2} }$$
and let $\Lambda=K\Gamma$. Then the generalized one point extension
algebra $E(\Lambda)$ has a basis $$\{e_1,\, e_2,\, e_3,\, \alpha,\,
\beta,\,\alpha^{\ast},\, \beta^{\ast}\}.$$ Define a linear mapping
on $E(\Lambda)$ by
$$
\begin{aligned} & \Theta(e_1)=\Theta(e_2)=\Theta(e_3)=0, & &
\Theta(\alpha)=\alpha^{\ast}, & & \Theta(\alpha^{\ast})=\alpha, &
\\& \Theta(\beta)=\beta+\beta^{\ast},
 & & \Theta(\beta^{\ast})=\beta-\beta^{\ast}.&
\end{aligned}
$$
Then a direct computation shows that $\Theta$ is a proper Jordan
derivation on $E(\Lambda)$.

On the other hand, we can also define two linear mappings $\Theta_1$
and $\Theta_2$ by
$$
\begin{aligned} & \Theta_1(e_1)=\Theta_1(e_2)=\Theta_1(e_3)=0, & &
\Theta_1(\alpha)=0, & & \Theta_1(\alpha^{\ast})=0, &
\\& \Theta_1(\beta)=\beta,
 & & \Theta_1(\beta^{\ast})=-\beta^{\ast}.&
\end{aligned}
$$ and
$$
\begin{aligned} & \Theta_2(e_1)=\Theta_2(e_2)=\Theta_2(e_3)=0, & &
\Theta_2(\alpha)=\alpha^{\ast}, & &
\Theta_2(\alpha^{\ast})=\alpha, &
\\& \Theta_2(\beta)=\beta^{\ast},
 & & \Theta_2(\beta^{\ast})=\beta.&
\end{aligned}
$$
It is easy to see that $\Theta_1$ is a derivation on $E(\Lambda)$
and $\Theta_2$ is an anti-derivation on $E(\Lambda)$. Therefore,
$\Theta$ is the sum of the derivation $\Theta_1$ and the
anti-derivation $\Theta_2$.
\end{example}

\bigskip

\noindent{\bf Acknowledgement} The first author would like to
express his sincere thanks to Chern Institute of Mathematics of
Nankai University for the hospitality during his visit. He also
acknowledges Prof. Chengming Bai for his kind consideration and
warm help.

\end{document}